\documentclass[11pt]{amsart}

\usepackage{amssymb,amsfonts}
\usepackage{euscript}

\usepackage[ps,matrix,arrow,curve,cmtip]{xy}

\title{Topological Modular Forms of Level $3$}
\date{September 10, 2008}
\author{Mark Mahowald}
\address{Department of Mathematics \\
Northwestern University \\
Evanston, IL}
\email{mark@math.northwestern.edu}
\author{Charles Rezk}
\address{Department of Mathematics \\
University of Illinois at Urbana-Champaign \\ 
Urbana, IL}
\email{rezk@math.uiuc.edu}

\thanks{The second author was partially supported by National Science
  Foundation grants DMS-0203936 and DMS-0505056.}

\numberwithin{equation}{section}

\makeatletter
  \let\c@subsection\c@equation
\makeatother

\theoremstyle{plain}   

\newtheorem{thm}[subsection]{Theorem}
\newtheorem{prop}[subsection]{Proposition}
\newtheorem{cor}[subsection]{Corollary}
\newtheorem{lemma}[subsection]{Lemma}

\theoremstyle{remark}
\newtheorem{rem}[subsection]{Remark}

\theoremstyle{plain}

\DeclareMathOperator{\id}{id}

\newcommand{\ra}{\rightarrow}

\newcommand{\xra}{\xrightarrow}

\DeclareMathOperator{\Ext}{Ext}

\newcommand{\len}[1]{\lvert#1\rvert}

\newcommand{\F}{\mathbb{F}}
\newcommand{\Z}{\mathbb{Z}}

\newcommand{\Q}{\mathbb{Q}}

\newcommand{\eev}{\wedge}
\newcommand{\sm}{\wedge} 

\newcommand{\dfn}{\textbf}

\newcommand{\ModuliEll}{\mathcal{M}}
\newcommand{\ModuliGenEll}{\overline{\mathcal{M}}}

\newcommand{\ModuliWeier}{\mathcal{M}_{\mathrm{Weier}}}

\newcommand{\Spec}{\operatorname{Spec}}

\newcommand{\tmf}{\mathrm{tmf}}
\newcommand{\TMF}{\mathrm{TMF}}

\newcommand{\MF}{\mathrm{MF}}
\newcommand{\mf}{\mathrm{mf}}
\newcommand{\String}{\mathrm{String}}

\newcommand{\bo}{\mathrm{bo}}
\newcommand{\bsp}{\mathrm{bsp}}
\newcommand{\Si}{\Sigma^\infty}
\newcommand{\CP}{\mathbb{CP}}
\newcommand{\sq}{\operatorname{Sq}}

\newcommand{\sstackN}{\mathcal{N}}
\newcommand{\Tot}{\mathrm{Tot}}

\def\elt{\circle*{3}}
\def\square{\square{1}}

\begin{document}

\begin{abstract}
We describe and compute the homotopy of spectra of topological modular
forms of level $3$.  We give some computations related to the
``building complex'' associated to level $3$ structures at the prime
$2$.   Finally, we note the existence of a number of connective models
of the spectrum $\TMF(\Gamma_0(3))$.  
\end{abstract}

\maketitle


\section{Introduction}

In this paper we collect a number of results related to the spectrum
of ``topological modular forms of level $3$'', denoted by
$\TMF(\Gamma_0(3))$.  We investigated this spectrum as part of an
attempt to describe the homotopy of the $K(2)$-local sphere (at the
prime $2$) in terms of modular forms.  We think this work will be most
useful in the context of a conjecture of Mark Behrens, which we
describe below.

Along the way, we considered a cosimplicial spectrum associated to a
certain ``building'', constructed using the moduli of certain kinds of
subgroups of elliptic curves.   The realization of this cosimplicial
spectrum is denoted $Q(3)$; it is discussed in
\S\ref{sec:building-complex}.  Behrens conjectures
\cite[1.6.1]{behrens-buildings-elliptic-curves} that there is a
cofiber sequence 
$$D_{K(2)}L_{K(2)}Q(3) \xra{Df} \tilde{S} \xra{f} L_{K(2)}Q(3)$$
where $D_{K(2)}$ denotes Spanier-Whitehead duality in the $K(2)$-local
category (at the prime $2$), and $\tilde{S}$ is a degree $2$ Galois
extension of the 
$K(2)$-local sphere.  More precisely,
$\tilde{S}=E_2^{h\tilde{\mathbb{G}}_2}$, where 
$\tilde{\mathbb{G}}_2=\tilde{\mathbb{S}}_2\rtimes \mathrm{Gal}$, with
$\tilde{\mathbb{S}}_2$ an index $2$ subgroup of the Morava stabilizer
group, the kernel of
$$\mathbb{S}_2 \xra{N} \Z_2^\times \ra (\Z/8)^{\times}/\{1,3\}.$$
Behrens has proved a version of this conjecture at the prime $3$
\cite{behrens-modular-description}.  

In this paper we give results related to the calculation of the
homotopy groups of $Q(3)$; we also describe how some known elements in
the homotopy groups of spheres are detected in it.  In doing so we
give a complete calculation of the homotopy of $\TMF(\Gamma_0(3))$.
We also describe some connected versions of $\TMF(\Gamma_0(3))$, which
appear naturally in the cobordism spectrum $M\String$.

The spectrum $\TMF(\Gamma_0(3))$, or one very much like it, has
appeared in a different context.  In
\cite{hu-kriz-real-oriented-homology}, the authors construct ``real''
versions of the Johnson-Wilson spectra $E(n)$; these are
$\Z/2$-equivariant ring spectra, which admit an orientation by the
``real'' complex bordism spectrum.  The homotopy fixed points of the
$\Z/2$-action on the $n$th real Johnson-Wilson spectrum is denoted
$ER(n)$.  As is clear from the work of
\cite{kitchloo-wilson-fibrations-real-spectra}, the spectrum $ER(2)$
is very much like the spectrum $\TMF(\Gamma_0(3))$.  In particular,
their calculation of $\pi_*ER(2)$ gives essentially the same answer as
our calculation of $\pi_*\TMF(\Gamma_0(3))$.  (Note that
$\TMF(\Gamma_0(3))$ and $ER(2)$
are not identical as ring spectra, because they are $\Z/2$-fixed
points of $\TMF(\Gamma_1(3))$ and $E(2)$ respectively, whose formal
groups are not isomorphic.  Presumably the construction
of \cite{hu-kriz-real-oriented-homology} can be carried out to
construct a ``real'' version of $\TMF(\Gamma_1(3))$.)

We would like to take this opportunity to dedicate this paper to
Professor Hirzebruch; in particular, we would like to express our
thanks for his book on modular forms
\cite{hirzebruch-manifolds-mod-forms}.

\section{Spectra of topological modular forms}

In this section we describe the examples of spectra of topological
modular forms  we are interested in.  
It will be convenient to use the
language of moduli stacks to identify them.  Thus, let $\ModuliEll$
denote the moduli stack of elliptic curves; an elliptic curve over a
base scheme $S$ is a smooth proper morphism $E\ra S$ whose geometric
fibers are elliptic curves.  Likewise, let $\ModuliGenEll$ denote the
moduli stack of generalized elliptic curves; this is a
compactification of $\ModuliEll$ obtained by ``adding the cusp''.
(See \cite{katz-mazur-arithmetic-moduli}.)  Thus, any morphism
$f\colon S\ra \ModuliGenEll$ 
determines a generalized elliptic curve $C_f\ra S$.  

There exists a line bundle $\omega\ra \ModuliGenEll$, associated to
the cotangent space of the identity section of a generalized elliptic
curve.  We write $\omega(C)$ for the line bundle over $S$ associated
to the generalized elliptic curve $C\ra S$.  When $C$ is smooth, we
can identify $\Gamma(S,\omega(C))$ with the set of invariant $1$-forms
on $C$.  A \dfn{modular form of level $1$ and weight $k$} is a section
of $\omega^{\otimes k}\ra \ModuliGenEll$.  Explicitly, a modular form
of weight $k$ is a function $g$ which associates to each pair $(C\ra
S, \eta)$ with $C$ a generalized elliptic curve over $S$, and $\eta\in
\Gamma(S,\omega)$, an element $g\in \mathcal{O}_S$, which is
compatible with base change and such that $g(C,\lambda
\eta)=\lambda^{-k}g(C,\eta)$ for $\lambda\in \mathcal{O}_S^\times$.
We write
$\mf_k=\Gamma(\ModuliGenEll,\omega^{\otimes k}$).

An \dfn{elliptic spectrum} \cite{ando-hopkins-strickland-thm-of-cube}
is a triple $(E,C,\phi)$ consisting of an even periodic ring spectrum
$E$, a generalized elliptic curve $C$ over $\pi_0 R$, and an
isomorphism $\phi$ from 
the formal group of $E$ to the formal completion of $C$ at the
identity.  (``Even periodic'' means: $\pi_1R=0$, and there exists
$u\in \pi_2R$ which is invertible in the graded ring $\pi_*R$.)  

All the examples of spectra of topological modular forms we need arise
from the theorem of Goerss-Hopkins-Miller.
\begin{thm}
There is a sheaf $\mathcal{O}_{\tmf}$ of $E_\infty$-ring spectra over the
stack $\ModuliGenEll$ in the \'etale topology, which has the following
property. 
For \'etale morphisms $f\colon \Spec(R)\ra \ModuliGenEll$, there is a
natural structure of elliptic spectrum $(\mathcal{O}_{\tmf}(f),C_f,\phi)$,
where $\pi_0 \mathcal{O}_{\tmf}(f) = R$ and $C_f$ is the generalized elliptic
curve over $R$ classified by $f$.
\end{thm}
As a consequence of this result, there is a spectral sequence
$$E_2=H^s(\mathcal{N},f^*\omega^{\otimes k}) \Longrightarrow
\pi_{2k-s}\mathcal{O}_{\tmf}(f)$$ associated to any \'etale map $f\colon
\mathcal{N}\ra \ModuliGenEll$.  The spectral sequence is functorial in $\mathcal{N}$.  If $\mathcal{N}=\Spec(R)$ is an affine
scheme, the spectral sequence collapses at $E_2$ and we have
$\pi_{2k}\mathcal{O}_{\tmf}(f)=\Gamma(\Spec(R),f^*\omega^{\otimes k})$.

Here are some of the basic examples we care about.
\begin{enumerate}
\item Let $\tmf$ denote the $(-1)$-connected cover of
  $\mathcal{O}_{\tmf}(\id\colon \ModuliGenEll\ra \ModuliGenEll)$.  This is
  the \dfn{connective spectrum of topological modular forms}, first
  constructed by Hopkins, 
  Miller, and the first author; some details are given in
  \cite{hopkins-mahowald-from-elliptic-curves}.  

\item Let $\TMF=\mathcal{O}_{\tmf}(\ModuliEll\ra \ModuliGenEll)$, the
  \dfn{periodic spectrum of topological modular forms}; the periodic
  invertible element is $\Delta^{24}\in \pi_{576}\TMF$.
\end{enumerate}

Another family of examples comes from introducing a level 3
structure.  Let $S$ be a scheme over $\Z[\tfrac{1}{3}]$.  If $C\ra S$
is an elliptic curve, let $C[3]$ denote the scheme of ``$3$-torsion
points'' of $C$ (that is, the kernel of $[3]\colon C\ra C$).  Then
$C[3]\ra S$ is an finite \'etale morphism (of degree $9$).  Locally in
the \'etale topology on $S$, $C[3]$ is isomorphic to the discrete
group scheme $\Z/3\times \Z/3$.

\begin{enumerate}
\item A \dfn{$\Gamma(3)$ structure} on $C\ra S$ is a choice of
  isomorphism $\Z/3\times \Z/3\ra C[3]$ of group schemes over $S$.
  Equivalently, a $\Gamma(3)$-structure is a choice of a pair of
  sections $s_1,s_2$ of $C[3]\ra S$ which are (locally in $S$)
  linearly independent.

\item A \dfn{$\Gamma_1(3)$ structure} on $C\ra S$ is a choice of
  monomorphism $\Z/3\ra C[3]$ of group schemes over $S$.
  Equivalently, a $\Gamma_1(3)$-structure is a choice of a (locally
  non-identity) section of   $C[3]\ra S$.

\item A \dfn{$\Gamma_0(3)$ structure} on $C\ra S$ is a choice of
  subgroup scheme $A\subseteq C[3]$ over $S$ which is isomorphic to
  $\Z/3$.  Equivalently, a $\Gamma_0(3)$-structure is a choice of
  equivalence class of $\Gamma_1(3)$-structures, where we identify
  sections which are carried to one another by the inversion
  $[-1]\colon C[3]\ra C[3]$.
\end{enumerate}
Each of these notions of level structure has an associated moduli
stack, and there are finite \'etale morphisms
$$\ModuliEll(\Gamma(3))\ra \ModuliEll(\Gamma_1(3))\ra
\ModuliEll(\Gamma_0(3)) \ra \ModuliEll[\tfrac{1}{3}]$$
of degrees $6$, $2$, and $4$ respectively.  Associated to these are
maps of $E_\infty$-rings $\TMF[\tfrac{1}{3}]\ra \TMF(\Gamma_0(3))\ra
\TMF(\Gamma_1(3))\ra \TMF(\Gamma(3))$.
Thus we obtain a spectrum of topological modular forms
$\TMF(\Gamma(3))$, and
similarly for $\Gamma(3)$ replaced with $\Gamma_1(3)$ and
$\Gamma_0(3)$.  All these spectra  are
$E_\infty$-rings under the commutative ring spectrum $S[\tfrac{1}{3}]$.

\begin{rem}
These notions of level structure admit generalizations to generalized 
elliptic curves over arbitrary schemes $S$ (i.e., without requiring
$3$ to be inverted, and allowing non-smooth curves).  However, the
resulting moduli stacks are not  
\'etale over $\ModuliGenEll$, and they won't play a role in this
paper.   See \cite{katz-mazur-arithmetic-moduli}. 
\end{rem}

Let $G=GL(2,\Z/3)$.  It is clear that this group acts on the set of
$\Gamma(3)$-structures of an elliptic curve, and thus acts on
$\TMF(\Gamma(3))$.  In particular, it is the Galois group of the
finite \'etale cover $\ModuliEll(\Gamma(3))\ra \ModuliEll$, and
thus we have that $\TMF[\frac{1}{3}] \approx \TMF(\Gamma(3))^{hG}$.
Let 
$$G_{\Gamma_1(3)} = \left\{ \begin{pmatrix} 1 & * \\ 0 & *
  \end{pmatrix} \in G \right\} \approx \Z/3\ltimes \Z/2$$ 
and
$$G_{\Gamma_0(3)} = \left\{ \begin{pmatrix} * & * \\ 0 & *
  \end{pmatrix} \in G\right\} \approx \Z/6 \ltimes \Z/2.$$
Then we have $\TMF(\Gamma_1(3))\approx
\TMF(\Gamma(3))^{hG_{\Gamma_1(3)}}$ and $\TMF(\Gamma_0(3))\approx
\TMF(\Gamma(3))^{hG_{\Gamma_0(3)}}$.    Furthermore, since
$G_{\Gamma_1(3)}$ is normal of index $2$ in $G_{\Gamma_0(3)}$, we have 
$\TMF(\Gamma_0(3))\approx \TMF(\Gamma_1(3))^{h\Z/2}$.

A \dfn{Weierstrass curve} over a ring $R$ is the closure
$C_{\underline{a}}= C_{(a_1,\dots,a_6)}$ in 
$\mathbb{P}^2_R$ of an affine curve of the form 
$$y^2+a_1\,xy+a_3\,y = x^3+a_2\,x^2+a_4\,x+a_6,$$
with $a_1,\dots,a_6\in R$.  A Weierstrass curve is smooth if and only
if the discriminant $\Delta=\Delta(a_1,\dots,a_6)$ is invertible in
$R$.  An isomorphism  $C_{\underline{a}'}\ra C_{\underline{a}}$
between Weierstrass curves is an algebraic map given by 
$$x\mapsto \lambda^{-2}x+r,\qquad y\mapsto
\lambda^{-3}y+\lambda^{-1}sx+t$$
which sends $C_{\underline{a}'}$ into $C_{\underline{a}}$.  There is
an Artin stack $\ModuliWeier$ of Weierstrass curves, determined by the
Hopf algebroid implicitly described above.

Every generalized elliptic curve $C\ra S$ admits a presentation,
locally over $S$ in the flat topology, as a
Weierstrass curve, in which the identity element of the elliptic curve
corresponds to the unique point at infinity on the Weierstrass curve.
Thus there is an open embedding
$\ModuliGenEll\ra \ModuliWeier$.  There is a line bundle $\omega$ over
$\ModuliWeier$ generated by invariant
differential
$$\eta_{\underline{a}}=\frac{dx}{2y+a_1\,x+a_3}=
\frac{dy}{3\,x^2+2a_2\,x+a_4-a_1\,y},$$  
and this pulls back to the line bundle $\omega$ over $\ModuliGenEll$.  

In particular, this means that the ring $\mf_*$ of level $1$-modular
is a subring of $A=\Z[a_1,a_2,a_3,a_4,a_6]$, and there are polynomials
$C_4, C_6\in A$ such that for any Weierstrass curve
$C_{\underline{a}}$, $c_i(C_{\underline{a}},\eta)=C(a_1,\dots,a_6)$.
These polynomials are those of \cite{deligne-formulaire}, \cite[p.\
46]{silverman-arith-ell-curves}.

\section{Modular forms of level $3$}

Explicit calculations about level $3$ structures flow from the
following observation.
\begin{prop}\label{prop:order-three-is-flex}
Let $C=C_{\underline{a}}\subset \mathbb{P}^2$ be a smooth Weierstrass
curve over a ring 
$R$ in which $3$ is invertible.  An $R$ point $P$ on $C$ has order $3$
if and only if it is a flex point; that is, if and only if the tangent
line $L$ at $P$  makes a triple intersection with $C$ at $P$.
\end{prop}
\begin{proof}
This is a consequence of the description of the group structure on a
smooth Weierstrass curve.  If $L\subset \mathbb{P}^2$ is a line and
$P_1,P_2,P_3$ its three points of intersection with $C$, counted with
multiplicity, then $[P_1]+[P_2]+[P_3]=0$ in the group structure of
$C$.  (See \cite{silverman-arith-ell-curves} or
\cite{katz-mazur-arithmetic-moduli}.)  
\end{proof}

We give a concrete description of $\ModuliEll(\Gamma_1(3))$.
\begin{prop}\label{prop:gamma-one-three-scheme}
Let $R$ be in which $3$ is invertible.  Let $(C,\eta,P)$ be a smooth
elliptic curve $C$ over $R$ together with an invariant $1$-form $\eta$
on $C$ and a $\Gamma_1(3)$-structure $P\in C(R)$.  Then there is a
unique triple 
$(C_1,\eta_1,P_1)$ 
and a unique isomorphism $(C_1,\eta_1,P_1)\approx (C,\eta,P)$ of this
data such that 
$$C_1: \quad y^2+A_1\,xy+A_3\,y=x^3, \qquad
\eta_1=\frac{dx}{2\,y+A_1\,x+A_3}=\frac{dy}{3\,x^2-A_1\,y}$$  
with $\Delta=A_3^3(A_1^3-27A_3)$ invertible in $R$, and such that $P_1$
is the point with $xy$-coordinates $(0,0)$.

In particular, the moduli problem of triples $(C,\eta,P)$ as above is
represented by the affine scheme $\Z[\tfrac{1}{3},A_1,A_3,\Delta^{-1}]$.
\end{prop}
\begin{proof}
Note that on the Weierstrass curve $C_1$, the tangent line to $P_1$ is
precisely the $x$-axis (so that $P_1$ is indeed a point of order $3$ by
\eqref{prop:order-three-is-flex}).  The result amounts to showing that
given $(C,\eta,P)$ where $C$ is a smooth Weierstrass curve, 
there is a \emph{unique} Weierstrass transformation sending $P$ to the
origin and sending the tangent line at $P$ to the $x$-axis.    This is
easiest to see in three steps.  Given a pair $(C,P)$ where $C$ is a
Weierstrass curve and $P=(\alpha,\beta)$ a point of exact order $3$,
use the variable substitution $x'=x+\alpha, y'=y+\beta$ to get a
Weierstrass curve $C'$ of the form
$$y^2+A_1\,xy+A_3\,y=x^3+A_2\,x^2+A_4\,x,$$
so that $P\in C$ corresponds with $P'=(0,0)\in C'$.  The tangent line
$L$ at $P'$ is given by $A_3y=A_4 x$, and since $L$ cannot be vertical
(every vertical line in the $xy$-plane intersects the curve at the
infinity), we have that $A_3$ is invertible.  Thus the transformation
$x''=x', y''=y'+(A_4/A_3)x'$ gives rise to a new Weierstrass curve
$C''$, which must have the desired form, since $P''=(0,0)$ must be a
triple intersection point of $C''$ with the $x$-axis.  Finally use a
transformation of the form $x'''=\lambda^{-2}x'',
y'''=\lambda^{-3}y''$ to get the invariant differential in the right form.
\end{proof}

\begin{cor}
We have 
$$\pi_*\TMF(\Gamma_1(3)) \approx H^0(\ModuliEll, \omega^{\otimes *})
\approx \Z[\tfrac{1}{3},a_1,a_3,\Delta^{-1}],$$
where $\len{a_1}=2$ and $\len{a_3}=6$.
\end{cor}
\begin{proof}
By \eqref{prop:gamma-one-three-scheme}, $H^*(\ModuliEll(\Gamma_1(3)),
\omega^{\otimes *})$ 
is equal to the cohomology of the Hopf algebroid with
$A=\Z[\tfrac{1}{3},a_1,a_3,\Delta^{-1}]$ and
$\Gamma=A[\lambda,\lambda^{-1}]$.  
\end{proof}

Note that for a curve of the form $y^2+a_1\,xy+a_3\,y=x^3$, the
inversion map is given by $[-1](x,y)=(x,-y-a_1\,x-a_3)$.  In
particular, the point $-P$ has $xy$-coordinates $(0,-a_3)$.

\begin{prop}\label{prop:gamma-0-3-ss-setup}
We have $H^s(\ModuliEll(\Gamma_0(3)),\omega^{\otimes k}) \approx
H^s(\Z/2, \Z[\tfrac{1}{3},a_1,a_3,\Delta^{-1}])$, where the generator
of $\sigma$ of $\Z/2$ acts by $\sigma(a_1)=-a_1$ and
$\sigma(a_3)=-a_3$.
In particular, there is a spectral sequence $E_r(\TMF(\Gamma_0(3)))$, with
$$E_2^{s,t}(\TMF(\Gamma_0(3))) \approx H^s(\Z/2,
\pi_t\TMF(\Gamma_1(3)))\Longrightarrow 
\pi_{t-s}\TMF(\Gamma_0(3)).$$ 
\end{prop}

Recall that the spectral sequence of \eqref{prop:gamma-0-3-ss-setup}
is that computing homotopy groups of the homotopy fixed point spectrum
$\TMF(\Gamma(3))^{hG_{\Gamma_0(3)}}$.  Thus there is a comparison map
$E_r(\TMF)\ra E_r(\TMF(\Gamma_0(3)))$ with the spectral sequence for
$\TMF=\TMF^{hG}$.  Associated to the 
natural map of commutative ring spectra $f^*\colon \TMF\ra
\TMF(\Gamma_0(3))$ there is a 
transfer map $f_!\colon \TMF(\Gamma_0(3))\ra \TMF$, which is a map of
$\TMF$-module spectra.  We note that
there is a map of spectral sequences $f_!\colon
E_r(\TMF(\Gamma_0(3)))\ra E_r(\TMF)$, and that on $E_2$-terms, this
map is precisely the cohomology transfer.  In particular, we have
\begin{prop}\label{prop:transfer-argument}
The composite $\TMF\xra{f^*} \TMF(\Gamma_0(3)) \xra{f_!} \TMF$ is
equal to $4\id$.
\end{prop}
\begin{proof}
Since both are maps of $\TMF$ modules, we only have to check the image
of $1\in \pi_0\TMF$.  It is a straightforward fact about group
cohomology that the composite $H^*(G,\pi_*\TMF(\Gamma(3))\ra
H^*(G_{\Gamma_0(3)},\pi_*\TMF(\Gamma(3)))\ra
H^*(G,\pi_*\TMF(\Gamma(3)))$ is given by multiplication by $4$.  The
result follows from the fact that $E_\infty^{s,s}(\TMF)=0$ for $s>0$.
\end{proof}

\section{Topological modular forms of level $3$}
\label{sec:homtopy-tmf-0-3}

The $E_2$-term can be described in terms of the bigraded ring
$R^{s,t}= \Z[\tfrac{1}{3},a_1,a_3,\zeta,\Delta^{-1}]/(2\zeta)$, where
$\zeta$ is given bidegree $(1,0)$.  If we assign ``odd'' weight to
$a_1$, $a_3$, and $\zeta$, then $E_2=H^s(\Z/2, \pi_t\TMF(\Gamma_1(3)))$
can be identified with the even part of $R^{s,t}$.

Let $x=\zeta a_3^3 \in E_2^{1,18}$.  There is an isomorphism
$$E_2\approx \Z[\tfrac{1}{3},a_1^2,a_1a_3,a_3^2,\Delta^{-1}][x]/(2x).$$
We write 
$$h_1= \zeta a_1 = x\, a_1a_3\, a_3^{-4}, \qquad h_2=\zeta^3
a_3=x^3\,a_3^{-8}, \qquad h_{2,0}=\zeta a_3 = x\, a_3^{-2}.$$  
We have the following $d_3$
differentials:
$$d_3\colon a_1^2\mapsto h_1^3,\qquad a_3^2\mapsto h_1h_{2,0}^2,\qquad
a_1a_3\mapsto 0,\qquad h_1\mapsto 0.$$
In general, if $c\in R^{0,t}$ has odd weight, then $d_3(c^2)=h_1
(\zeta c)^2$; this can be proved by a cup-$1$ construction.
We can
identify $h_1$ as the class representing the  
image of $\eta\in \pi_1S^0$.  We will show in
\eqref{cor:a1a3-is-permanent-cycle} that $a_1a_3$ is a permanent cycle.

This forces:
$$d_3\colon h_{2,0}\mapsto h_1h_{2,0}\zeta^2,\qquad x\mapsto 0.$$
(The differentials already computed show that $d_3\colon E_3^{4,20}\ra
E_3^{7,22}$ is injective, so $d_3(x)=0$, and $h_{2,0}=xa_3^{-2}$.) 

We note that the transfer argument of \eqref{prop:transfer-argument}
shows that whenever there is an element $\alpha$ of order $8$ in
$\pi_*\TMF$, its image $\alpha'$ in $\pi_*\TMF(\Gamma_0(3))$ is
non-trivial.  Furthermore, if $\alpha$ and $4\alpha$ are detected on the 
$s_0$-line and $s_1$-line of
$E_\infty(\TMF)$, respectively, then $\alpha'$ must be detected on the
$s$-line of $E_\infty(\TMF(\Gamma_0(3)))$, where $s_0\leq s\leq s_1$.
This allows us to see that 
$$h_2=\zeta^3 a_3 \in E_2^{3,6}\qquad\text{detects}\qquad \nu \in
\pi_3\TMF(\Gamma_0(3)),$$ 
and
$$h_{2,0}^4=\zeta^4a_3^4 \in E_2^{4,24}\quad\text{detects}\quad
\bar{\kappa}\in \pi_{20}\TMF(\Gamma_0(3)),$$ 
where these elements are images of the like-named classes in $\pi_*S^0$.

At this point there are no possible differentials until the
$E_7$-term.  There is a map $E_7\ra \F_2[\Delta,\Delta^{-1},x]$ which
is surjective, and is an isomorphism on lines $s\geq3$.

The element $h_{2,0}^4=x^4\Delta^{-1}$ is the image of the class
representing $\bar{\kappa}$ in the spectral sequence for $\pi_*\TMF$, so it
is a permanent cycle.  There is a relation $\bar{\kappa}^6=0$ in
$\pi_*\TMF$, and hence this relation must hold in
$\pi_*\TMF(\Gamma_0(3))$.  The only possible differential which can do
this is is $d_7(x^{17}\Delta^{-7})=(h_{2,0}^4)^6$.   

This implies that $d_7$ is non-trivial on either $x$ or on $\Delta^4$
(but not both).  In either case, one sees that $E_7=E_\infty$, and
that this vanishes for lines $s\geq7$.  

Thus, we see that $x$ is a permanent cycle, whence
$d_7(\Delta)=h_{2,0}^4\nu$.    
\begin{prop}
The above provides a description of $\pi_*\TMF(\Gamma_0(3))$.  In
particular, there is an exact sequence
  \begin{multline*}
    0\ra \F_2[\Delta^{\pm 2}]\{\nu, \nu^2, x, \eta x,\bar{\kappa},
    x^2, \nu 
    x^2\} \ra \pi_*\TMF(\Gamma_0(3)) 
\\
    \ra \bo_*[\tfrac{1}{3},\Delta^{\pm 1}]\{1,a_1a_3\} \oplus
    \bsp_*[\tfrac{1}{3},\Delta^{\pm1}]\{2a_3^2,2(a_1a_3)a_3^2\} \ra
    \Delta\,\F_2[\Delta^{\pm 2}]\ra 0.
  \end{multline*}
The spectrum $\TMF(\Gamma_0(3))$ is $48$-periodic, with periodicity
generated by $\Delta^2$.  The elements $\eta$, $\nu$, and
$\bar{\kappa}$ are images of the like-named elements in the $1$, $3$
and $20$ stems of $S^0$.  The element $x$ lies in
$\pi_{17}\TMF(\Gamma_0(3))$.  Furthermore, this sequence encodes the
multiplicative structure, except that in addition we must note that
$a_1^4x=\eta(a_1a_3)^3$, $a_1a_3\, x= \eta\Delta$, $a_1a_3\,x^2=
\eta^2(a_1a_3)a_3^4$, 
$\nu
x=\bar{\kappa}$, and that 
$x^3=\nu\Delta^2$, $x^4=\bar{\kappa}\Delta^2$, $x^5 =\nu x \Delta^2$,
$x^6=\nu^2\Delta^4$, and $x^7=0$.
\end{prop}

\begin{rem}
It is notable how element $x$ generates all the $v_1$-periodic torsion
in $\pi_*\TMF(\Gamma_0(3))$. It is particularly nice since $x$ and $x^2$ are
$v_1^4$ periodic.  An elementary calculation on homotopy
groups shows that 
there is a cofiber sequence
$$\Sigma^{17}\TMF(\Gamma_0(3)) \xra{x\cdot} \TMF(\Gamma_0(3))
\ra \TMF(\Gamma_1(3))$$
in the category of $\TMF(\Gamma_0(3))$-module spectra.
This seems to be an example of one of the fibrations produced in
\cite{kitchloo-wilson-fibrations-real-spectra}, where such fibrations
are constructed for all the ``real'' Johnson-Wilson spectra.
\end{rem}

\section{Hecke operators, and the building complex}
\label{sec:building-complex}

In this section we describe a version of the ``building complex''
which has been studied fruitfully by Behrens
\cite{behrens-modular-description}.  In the following, we assume that
all schemes $S$ are defined over $\Z[\tfrac{1}{3}]$.

We define a semi-simplicial stack $\sstackN_\bullet$, together with a ``line
bundle'' $\omega_\bullet$ over $\sstackN_\bullet$,  as follows.  Let
$\sstackN_k$ be the moduli stack of data of the form
$$C_0\xra{\phi_1} C_1 \xra{\phi_2} \cdots \xra{\phi_k} C_k,$$
where the $C_i$ are smooth elliptic curves over a base $S$, and the
$\phi_i$ are isogenies of elliptic curves which are not isomorphisms,
such that $\ker(\phi_k\cdots\phi_1)\subseteq C_0[3]$.  In
particular, each $\phi_i$ has degree either $3$ or $3^2$.  Face maps
are defined in the evident way.  It is readily apparent that
$\sstackN_k$ is empty for $k>2$.

The
semi-simplicial stack can be completed to a simplicial stack by
formally adding degeneracies (which amounts to allowing some of the $\phi_i$
to be isomorphisms).

To give the line bundle $\omega_\bullet$,  we set $\omega_k$ to be the
line bundle over $\sstackN_k$ defined by $\omega_k=\omega(C_0)$,
together with  ``descent data'' specified by  the following isomorphisms
\begin{align*}
(d_j^* \omega_{k-1} \xra{\sim} \omega_k) &= (\id\colon \omega(C_0)\ra
\omega(C_0)) & \text{for $0<j\leq k$,}
\\
(d_0^* \omega_{k-1} \xra{\sim} \omega_k) &= (\phi_1^*\colon
\omega(C_1)\ra \omega(C_0)),
\end{align*}
where $\phi_1^*\colon \omega(C_1)\ra \omega(C_0)$ is the map induced
by pulling back $1$-forms; since $3$ is inverted in the ground ring,
these are isomorphisms.

We now describe all the structure here in terms of more familiar
objects, making use of the fact that the moduli of isogenies $C_0\ra
C_1$ is equivalent to the moduli of pairs $(C_0,A)$, consisting of a
curve $C_0$ and a finite subgroup $A$.  Thus, consider the following
morphisms of moduli stacks.  
\begin{align*}
h\colon & \ModuliEll\ra \ModuliEll, && (C)\mapsto (C/C[3]),
\\
f\colon & \ModuliEll(\Gamma_0(3))\ra \ModuliEll, &&  (A<C)\mapsto (C),
\\
q\colon & \ModuliEll(\Gamma_0(3))\ra \ModuliEll, &&  (A<C)\mapsto (C/A),
\\
t\colon & \ModuliEll(\Gamma_0(3))\ra \ModuliEll(\Gamma_0(3)), &&
(A<C) \mapsto (C[3]/A<C/A). 
\end{align*}
We have described the morphisms in terms of the effect on objects.
For instance, $q\colon \ModuliEll(\Gamma_0(3))\ra \ModuliEll$ is the
morphism of stacks associated to the operation which sends the data of
a smooth elliptic curve $C$, together with a subgroup scheme $A<C$
locally isomorphic to $\Z/3$, and produces the quotient group scheme
$C/A$, which is again an elliptic curve.  The notation $C[3]$ denotes
the subgroup of $3$-torsion points in $C$; under our hypotheses on the
base scheme, $C[3]$ is an etale group scheme locally isomorphic to
$\Z/3\times \Z/3$.  

We may now describe the semi-simplicial stack $\sstackN_\bullet$ by the
following picture.
$$\xymatrix{
& {\ModuliEll(\Gamma_0(3))} \ar@<1ex>[dl]|{d_1=f} \ar@<-1ex>[dl]|{d_0=q}
\\
{\ModuliEll}
& {\amalg} 
& {\ModuliEll(\Gamma_0(3))} \ar@<1ex>[ul]|{d_2=\id} \ar@<-1ex>[ul]|{d_0=t} \ar[dl]|{d_1=f}
\\
& {\ModuliEll} \ar@<1ex>[ul]|{d_1=\id} \ar@<-1ex>[ul]|{d_0=h}
\\
{\mathcal{N}_0}
& {\mathcal{N}_1}
& {\mathcal{N}_2}
}$$
The simplicial identities follow from the identities 
\begin{equation}\label{eq:stack-identities}
q=ft\qquad \text{and} \qquad qt=hf.
\end{equation}  The line bundle $\omega_\bullet$ is described by
$$\omega_0=\omega,\qquad
\omega_1|_{\ModuliEll(\Gamma_0(3))}=f^*\omega,\qquad
\omega_1|_{\ModuliEll}=\omega, \qquad \omega_2=f^*\omega,$$
with the non-trivial parts of the descent data given by the maps
$q^*\omega\xra{\sim} f^*\omega$ and $h^*\omega\xra{\sim} \omega$
induced by isogenies.

There is an associated semi-cosimplicial commutative ring spectrum
$$\xymatrix{
& {\TMF(\Gamma_0(3))} \ar@<-1ex>[dr]|{d^2=\id} \ar@<1ex>[dr]|{d^0=t^*} 
\\
{\TMF} \ar@<-1ex>[ur]|{d^1=f^*} \ar@<1ex>[ur]|{d^0=q^*}
\ar@<-1ex>[dr]|{d^1=\id} \ar@<1ex>[dr]|{d^0=h^*}
& {\times}
& {\TMF(\Gamma_0(3))} 
\\
& {\TMF} \ar[ur]|{d^1=f^*}
}$$
Following Behrens, we write $Q(3)$ for the geometric realization of
the cosimplicial ring

There is also a cosimplicial bigraded ring given by
$H^*(\mathcal{N}_\bullet,\omega^{\otimes *})$; we also write
$f^*,q^*,t^*,h^*$ for the induced maps on cohomology.  In this
section we compute the effect of these maps.

Write $\MF_k=H^0(\ModuliEll,\omega^{\otimes k})$ and
$\MF(\Gamma_0(3))_k = H^0(\ModuliEll(\Gamma_0(3)), \omega^{\otimes
k})$ for the rings of modular forms.  Recall that
\begin{align*}
\MF_* &=\Z[c_4,c_6,\Delta, \Delta^{-1}]/(c_4^3-c_6^2-1728\Delta),
\\
\MF(\Gamma_0(3))_* &=
\Z[\tfrac{1}{3},a_1^2,a_3^2,a_1a_3,(a_1^3a_3^3-27a_3^4)^{-1}].
\end{align*}

\begin{prop}\label{prop:formulae}
The maps $f^*\colon \MF\ra \MF(\Gamma_0(3))$, $h^*\colon \MF\ra
\MF$, $q^*\colon \MF\ra \MF(\Gamma_0(3))$, and $t^*\colon
\MF(\Gamma_0(3))\ra \MF(\Gamma_0(3))$, are described by
\begin{align*}
  f^*(c_4) &= a_1^4-24\,a_1a_3,
&  h^*(c_4) &= 3^4 \,c_4,
\\
f^*(c_6) &= -a_1^6+36\,a_1^3a_3-216\,a_3^2,
& h^*(c_6) &= 3^6\, c_6,
\\
f^*(\Delta) &= a_1^3a_3^3-27\,a_3^4,
& h^*(\Delta) &= 3^{12}\, \Delta.
\end{align*}
\begin{align*}
q^*(c_4) &= a_1^4+216\,a_1a_3,
\\
q^*(c_6) &= -a_1^6+540\,a_1^3a_3+5832\,a_3^2,
\\
q^*(\Delta) &= a_1^9a_3-81\,a_1^6a_3^2 +2187\,a_1^3a_3^3 -19683\,a_3^4.
\end{align*}
\begin{align*}
t^*(a_1^2) &= -3\,a_1^2,
\\ 
t^*(a_1a_3) &= \tfrac{1}{3}\, a_1^4-9\,a_1a_3,
\\
t^*(a_3^2) &= -\tfrac{1}{27}\,a_1^6 +2\, a_1^3a_3 -27\, a_3^2.
\end{align*}
\end{prop}
We will give the proof of \eqref{prop:formulae} at the end of the section.

Let $A=\Z[\tfrac{1}{3},a_1,a_3,\Delta^{-1}]$; let $C$ denote the curve
given by Weierstrass equation $y^2+a_1\,xy+a_3\,y=x^3$, and let
$\eta=\frac{dx}{2y+a_1\,x+a_3}$ denote the usual invariant
$1$-form.  Thus $C$ is a model for the universal curve over
$\ModuliEll(\Gamma_1(3))$, with $P_0=(0,0)$ as the distinguished point
of order $3$, and $-P_0=(0,-a_3)$.

\begin{prop}\label{prop-isogeny-construction}
Let $C'$ denote the Weierstrass curve over $A$ defined by the affine
equation
$$Y^2 + a_1XY + 3a_3Y = X^3 -6a_1a_3X -(9a_3^2+a_1^3a_3),$$
with non-vanishing $1$-form $\eta'=dX/(2Y+a_1X+3a_3)$.
There is an isogeny $\phi\colon C\ra C'$ of degree $3$ defined by
$$\phi\colon (x,y) \mapsto (X,Y) =
\left(x-\frac{a_3y}{x^2}+\frac{a_3x}{y}, y-\frac{a_3^2y}{x^3} 
- \frac{a_3x^3}{y^2}\right),$$
and under this map $\phi^*\eta'=\eta$.
The kernel of this isogeny is precisely the subgroup $A$ of order $3$
generated by  $P_0\in C$
\end{prop}
\begin{proof}
It is straightforward
to check that the function $\phi$ defines a rational map of curves $C\ra
C'$ with $\phi^*\eta' = \eta$, for example using a computer algebra
package. 
One calculates that the discriminant of the curve $C'$ is $\Delta' =
a_3(a_1^3-27a_3)^3$, and hence $C'$ is a non-singular elliptic curve
in Weierstrass form.  Thus one concludes
that $\phi$ is a non-singular map between smooth elliptic curves.  Since the
coordinate functions $X(x,y)$ and $Y(x,y)$ have poles only on the
subgroup $A$ of $E$ we see that $A$ is precisely the kernel of the map.
\end{proof}

\begin{rem}
The curve $C'$ was obtained by the following procedure.  First,
consider the map $\sigma\colon C\ra C$ defined using the group
structure on $E$ by $\sigma(P) = P+P_0$ where $P_0=(0,0)$ is a
generator of $A$.  One
computes that $\sigma(x,y) = (-a_3y/x^2, -a_3^2y/x^3)$.  Then $X = x +
\sigma^*x + (\sigma^*)^2x$ and $Y=y + \sigma^*y + (\sigma^*)^2y$ must
be Weierstrass coordinates for the quotient variety $E'$ and hence
satisfy a Weierstrass polynomial which one can solve for explicitly;
this polynomial is the equation for $C'$.
\end{rem}

\begin{proof}[Proof of \eqref{prop:formulae}]
We  compute the effect of the maps $f^*$ and $q^*$ on modular
forms, using the fact that the isogeny $C\ra C'$ of
\eqref{prop-isogeny-construction} exhibits the universal example of an
isogeny of degree $3$.  Thus, the map $f^*\colon \MF\ra
\MF(\Gamma_0(3))$ sends $c_i\in \MF_k$ ($i=4,6$) to 
$$(f^*c_i)(\phi\colon C\ra C', \eta) = c_i(C,\eta),$$
which can be
read off from formulas found in \cite[p.\ 
46]{silverman-arith-ell-curves} which express the $c_i$s as
polynomials $C_i(a_1,a_2,a_3,a_4,a_6)$ in 
the Weierstrass parameters $a_1,a_2,a_3,a_4,a_6$; in this case, we 
have $f^*c_i=C_i(a_1,0,a_3,0,0)$.
The map $q^*\colon \MF\ra \MF(\Gamma_0(3))$ is described by
$$(q^*c_i)(\phi\colon C\ra C',\eta) = c_i(C',\eta'),$$
so that $q^*c_i=C_i(a_1,0,3a_3,-6a_1a_3, -(9a_3^2+a_1^3a_3))$.

The effect of $h^*$ is computed using the $3$-power isogeny $[3]\colon
C\ra C$.  This isogeny acts on invariant $1$-forms by
$[3]^*\eta=3\eta$, so that 
$$(h^*c_i)([3]\colon C\ra C,\eta) = c_i(C,\tfrac{1}{3}\eta) =
3^i\,c_i(C,\eta).$$ 

The formula for $t^*$ is obtained from the others, using the
identities \eqref{eq:stack-identities}, and the fact that
$\MF(\Gamma_0(3))$ is an integral 
domain.   Thus, we have
$$t^*(10a_1^4)=t^*(9f^*(c_4)+q^*(c_4)) = 9q^*(c_4)+f^*(h^*(c_4))=
90a_1^4,$$
from which we see that $t^*(a_1^2)=\epsilon\,3a_1^2$ where $\epsilon
\in \{\pm 1\}$.  We also have
$$t^*(240a_1a_3)=t^*(q^*(c_4)-f^*(c_4))=f^*(h^*(c_4))-q^*(c_4)=
80a_1^4-2160a_1a_3,$$
whence $t^*(a_1a_3)=\frac{1}{3}a_1^4-9a_1a_3$.  The identity
$t^*((a_1a_3)^2)= t^*(a_1^2)t^*(a_3^2)$ gives $t^*(a_3^2)=\epsilon
(-\frac{1}{27}a_1^6 + 2a_1^3a_3-27a_3^2)$.  Finally, we can use the
identity $t^*f^*(c_6)=q^*(c_6)$ to show that $\epsilon=+1$.
\end{proof}

\section{Some homotopy classes detected by $Q(3)$}  
 
In this section we describe how the image of $J$ classes are detected
in $Q(3)$; this will give a proof of
\eqref{cor:a1a3-is-permanent-cycle} which we needed for the
computation of $\pi_*\TMF(\Gamma_0(3))$.  

Recall that $\TMF$ admits an $E_\infty$ $\String$ orientation
\cite{ando-hopkins-rezk-string-orientation} which refines the Witten
genus.   We use this to detect 
elements in the image of $J$ in the homotopy of $Q(3)$, using the
standard formalism which we review below.

Let $\gamma \colon X\ra B\String$ be the map classifying a stable
string bundle over a pointed space $X$.   Write $T(\gamma)$ for the
Thom spectrum of $\gamma$, and define $\bar{T}(\gamma)$ and $\alpha$
by the cofiber sequence
$$\Sigma^{-1}\bar{T}(\gamma) \xra{\alpha} S^0 \ra T(\gamma) \ra
\bar{T}(\gamma).$$  
Let $R^\bullet$ denote a cosimplicial $E_\infty$-ring spectrum, such
that $R^0$ admits a string orientation, and let $Q=\Tot R^{\bullet}$.
Then the composite $\Sigma^{-1}\bar{T}(\gamma)\xra{\alpha} S^0 \ra Q$
is detected by an element $e(\alpha)\in [\bar{T}(\gamma), R^1]$,
defined by a diagram
$$\xymatrix{
{\Sigma^{-1}\bar{T}(\gamma)} \ar[r]^-{\alpha} 
& {S^0} \ar[r] \ar[d]_{1}
& {T(\gamma)} \ar[r] \ar[d]^t
& {\bar{T}(\gamma)} \ar[d]^{e(\alpha)}
\\
& {Q} \ar[r]
& {R^0} \ar[r]_{d^0-d^1}
& {R^1}
}$$
where $t$ is the map determined by the string orientation of $R^0$.

If $X=\Sigma Y$, then $\bar{T}(\gamma)=\Si \Sigma Y$ and 
$\alpha=J\tilde{\gamma}\colon \Si Y\ra S^0$, where
$\tilde{\gamma}\colon Y\ra \String$ is the adjoint of $\gamma$.

\begin{prop}
Let $\gamma_{2n}\colon S^{4n}\ra B\String$ be the standard generator
($n\geq2$), so that $\alpha_{2n}=J\tilde{\gamma}_{2n}\colon S^{4n-1}\ra S^0$ is
the generator of the image of $J$.  Then $e(\alpha_{2n})\colon S^{4n}\ra
R^1\ra R^1_{\Q}$ is given by 
$$u_n\cdot((d^0)_*(b_{2n})-(d^1)_*(b_{2n})) \in
\pi_{4n}R^1\otimes \Q,$$
where 
$$Q(x)=\frac{x}{\exp_F(x)} = \exp\left[ 2\sum_{k\geq1} b_{2k}
  \frac{x^{2k}}{{2k}!} \right] \in H^*(\CP^\infty, \pi_* R^0\otimes \Q)$$
is the Hirzebruch series associated to the given string orientation of
$R^0$, and $u_n=1$ if $n$ is even and $u_n=2$ if $n$ is odd.
\end{prop}
\begin{proof}
This is ``standard''; a proof appears, for instance, in
\cite{ando-hopkins-rezk-string-orientation}. 
\end{proof}

In particular, taking $R^\bullet$ to be the cosimplicial ring
associated to the building complex of \S\ref{sec:building-complex},
and $Q=Q(3)=\Tot (R^\bullet)$, we see that $e(\alpha_{2n})\in
\pi_{4n}R^1 \approx \pi_{4n}\TMF(\Gamma_0(3))\times \pi_{4n}\TMF$ is a
class which modulo torsion has the name
$$(u_{n}\cdot (q^*G_{2n}-f^*G_{2n}), u_n\cdot (3^{2n}-1)G_{2n}),$$
where $G_{2n}\in \mf_{2n}\otimes \Q$ is the unnormalized Eisenstein
series with $q$-expansion
$$G_{2n}(q) = -\frac{B_{2n}}{4n} + \sum_{m\geq0} q^m \sum_{d|m}
d^{2n-1}.$$

\begin{cor}\label{cor:a1a3-is-permanent-cycle}
There is an element in $\pi_8\TMF(\Gamma_0(3))$ which maps to
$a_1a_3\in \pi_8\TMF(\Gamma_1(3))$.  
\end{cor}
\begin{proof}
The element is given by $e(\gamma_4)\colon S^8\ra \TMF(\Gamma_0(3))$,
since $240\,G_4=c_4$, so that
$q^*G_4-f^*G_4=\tfrac{1}{240}[(a_1^4+24\,a_1a_3)-(a_1^4-216\,a_1a_3)]
= a_1a_3$.
\end{proof}

A number of elements in $\pi_*S^0$ are detected in $\pi_*Q(3)$; we
hope to provide calculations of these in a future paper.

\section{Connective models for $\TMF(\Gamma_0(3))$}

By a ``connective model'' of a $v_2$-periodic ring spectrum $R$, we mean a
connective spectrum $X$ such that $L_{K(2)}X\approx L_{K(2)}R$.  More
optimistically, we can ask that $L_{K(1)\vee K(2)}X\approx
R^{\eev}_p$.  Even more optimistically, we may hope that $X$ is a ring
spectrum, or even an $E_\infty$-ring.

Thus, the periodic spectrum of topological modular forms
$\TMF$ comes with a canonical connective model $\tmf$, which is itself
an $E_\infty$-ring.  There seems to be no known construction of a
connective $\tmf(\Gamma_0(3))$ which is also an $E_\infty$-ring.

In this section, all results are at the prime $2$.  Cohomology refers
to mod $2$ cohomology.  
We will make much use of the fact that $H^*\tmf\approx
A\otimes_{A(2)}\Z/2$. 

Our first result is to construct a simple finite complex  and a map
of this complex into $\TMF(\Gamma_0(3))$ such that the extension over
$\TMF$ is a weak equivalence.  We will think of this 
as a recognition complex.  Then if we can map this complex into other
naturally occurring module spectra we can determine if such a spectrum
is a connective model for $\TMF(\Gamma_0(3))$.  

We begin with a finite complex $X=\bo_1$, which is the $7$-skeleton of
$\bo$; that is, $X=S^0\cup_\nu e^4\cup_\eta e^6\cup_{2\iota} e^7$.
Let $Y=\Sigma^7 D\bo_1$; that is, $Y$ is the Spanier-Whitehead dual of
$\bo_1$, shifted so that the bottom cell is in dimension $0$.

\begin{thm}
There is a map $f\colon \Sigma^6X\ra Y$ so that the composite $S^6\ra
\Sigma^6X\ra Y$ is $\nu^2$ on the bottom cell.
\end{thm}
\begin{proof}
The proof is straightforward after computing $\pi_*Y$ through
dimension $14$.  The following is the Adams $E_2$-term in the usual
way of presenting such charts.
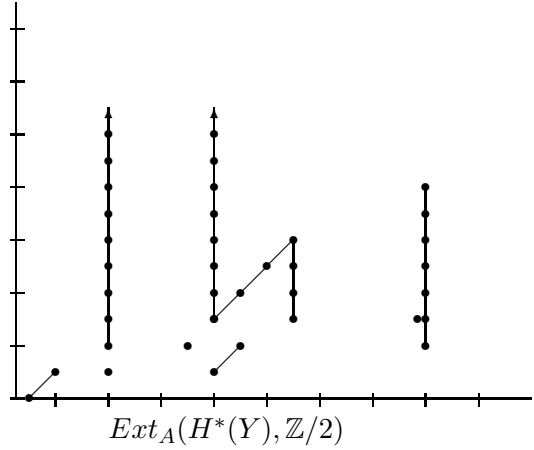
\begin{center}
\begin{picture}(200,200)
\put(10,10){\line(1,0){190}}
\multiput(20,7)(20,0){9}{\line(0,1){5}}
\put(5,10){\line(0,1){150}}
\multiput(3,10)(0,20){8}{\line(1,0){5}}
\multiput(10,10)(10,10){2}{\elt}
\put(10,10){\line(1,1){10}}
\multiput(40,20)(30,10){2}{\elt}
\put(40,30){\vector(0,1){90}}
\multiput(40,30)(0,10){9}{\elt}
\multiput(80,20)(10,10){2}{\elt}
\put(80,20){\line(1,1){10}}
\put(80,40){\vector(0,1){80}}
\put(80,40){\line(1,1){30}}
\multiput(80,40)(0,10){8}{\elt}
\multiput(80,40)(10,10){4}{\elt}
\put(110,40){\line(0,1){30}}
\multiput(110,40)(0,10){3}{\elt}
\put(160,30){\line(0,1){60}}
\multiput(160,30)(0,10){7}{\elt}
\put(157,40){\elt}
\put(40,-5){$Ext_A(H^*(Y),\Z/2)$}
\end{picture}
\end{center}

\bigskip

From this chart it is easy to see that the class which begins the map
$\Sigma^6X\ra Y$ sends the bottom class to a class in the six stem.
Clearly, $\nu$ on this class is zero and this allows an extension over
the four skeleton of $X$.  The last Moore space maps in dimensions
$11$ and $12$, both of which are zero. 
\end{proof}

The following result is our recognition principle.  Let $Z=Y\cup_f
C\Sigma^6X$.  
\begin{prop}\label{prop:recognition-complex}
There is a map $g\colon \Sigma^{17}Z\ra \TMF(\Gamma_0(3))$ whose
extension to $\Sigma^{17}Z\sm \TMF\ra \TMF(\Gamma_0(3))$ is a weak
equivalence. 
\end{prop}
\begin{proof}
This follows from the calculations of $\pi_*\TMF(\Gamma_0(3))$ we made
in \S\ref{sec:homtopy-tmf-0-3}.  In particular, $g$ is constructed so
that on the bottom cell it is given by $x\colon S^{17}\ra
\TMF(\Gamma_0(3))$.  
\end{proof}

We will use this result to show that various naturally occurring
spectra $X$ are copies of covers of $\TMF(\Gamma_0(3))$.  In each case, a
complete homotopy calculation of $X$ looks like a connective version of
$\TMF(\Gamma_0(3))$; this allows the construction of a map from $Z$ to $X$
just as in the proof of \eqref{prop:recognition-complex}.  When $X$ is
a $\tmf$-module, the map extends to a map $Z\sm \tmf$ to $X$.

The first example is the bit in $M\String$ which begins in dimension
$16$.  Recall that the cohomology of $M\String$ is free over
$A\otimes_{A(2)} \Z/2$.  In dimension $16$ there is a free generator
and starting in dimension $20$ there begins an extended $A(2)$-module
$A(2)/(\sq^1,\sq^5,\sq^6)$.   It was determined in
\cite{davis-mahowald-ext-a2-stunted-projective} that these two pieces
are connected by 
$2\nu$.  The homotopy of the resulting $\tmf$ module was computed in
\cite{gorbounov-mahowald-mo8}.  An inspection of that calculation shows
that this module is just a connected version of
$\pi_*\TMF(\Gamma_0(3))$.  

We write $\tmf(\Gamma_0(3))$ for this connected model of
$\TMF(\Gamma_0)$.  It is of special interest,
since it has a good chance of being a ring spectrum.  The homotopy of
$\pi_*\tmf(\Gamma_0(3))$ is described in
\cite{gorbounov-mahowald-mo8}; in terms of the spectral sequence calculation of
$\pi_*\TMF(\Gamma_0(3))$ given in \S\ref{sec:homtopy-tmf-0-3}, the
homotopy of $\tmf(\Gamma_0(3))$ corresponds to the part of the
$E_2^{s,t}$-term for which (i) $s\leq t-s$ (i.e., below the line of slope
one), and (ii) whose $E_1$-names only include non-negative powers of
the element $\Delta$.


\begin{rem}
The module $\Z/2\oplus \Sigma^4 A(2)/(\sq^1,\sq^5,\sq^6)$ admits the
following nice description: it is isomorphic as a module to
$H^*(\bo_1)^{\otimes 2}/\Sigma_2$,  the symmetric coinvariants of
$H^*\bo_1$.
\end{rem}

In \cite{gorbounov-mahowald-mo8} it is also shown that there is an
extended $A(2)$-module $M$ beginning in dimension $44$ which is 
$$A(2)\{a_{44},a_{49}\}/(\sq^1 a_{44},\,\sq^5a_{44},\,
\sq^7a_{44}+\sq^2 a_{49},\, \sq^4 a_{49}+\sq^6\sq^3 a_{44}).$$
The Adams spectral sequence for this extended $A(2)$-module is easy to
compute.  Let $M_1=A(2)\{a_5\}/A(2)\{\sq^2a_5,\,\sq^4a_5\}$, and
$M_2=A(2)\{a_0\}/A(2)\{\sq^1a_0,\,\sq^5a_0,\,\sq^{13}a_0\}$.  Then we
have a short exact sequence
$$M_2\ra M\ra M_1.$$
The $\Ext$ chart for $M_1$ is computed in
\cite{davis-mahowald-ext-a2-stunted-projective}.  The module $M_2$
fits into an exact sequence
$$M_2\ra A(2)/(\sq^1,\sq^5) \ra
A(2)\{a_{13}\}/(\sq^1,\sq^2,\sq^{10}).$$
The $\Ext$-chart for the right module is computed in
\cite{davis-mahowald-ext-a2-stunted-projective} while the middle
module is just $H^*\bsp$.  Thus all parts are easily computed and
there is a non-trivial connecting homomorphism in the sequence to
compute $\Ext_{A(2)}(M)$.  Again there are no possible Adams
differentials.  Its homotopy agrees with the $11$ homotopy skeleton of
the fiber of a map from $\tmf(\Gamma_0(3))\ra \bo\vee \Sigma^8\bo$.
The first map picks up the unit and the second picks up $a_1a_3$.

These two results suggest that a possible path toward understanding
$M\String$ is via these module maps and others like them.  We have the
orientation now thanks to Ando, Hopkins, and Rezk.  

The $A(2)$-module structure of $A\otimes_{A(2)}\Z/2$ is known.  It is
most easily described in terms of $\bo$-Brown-Gitler spectra.  Compare
\cite{goerss-jones-mahowald-generalized-bg-spectra}.  The first two
after the unit are $\Sigma^8\bo_1 = \Sigma^8 X$ and
$\Sigma^{16}\bo_2$.  Because of a differential in the Adams spectral
sequence, these two pieces are connected.  We have the following:
\begin{prop}
$K(2)$-locally, there is a cofiber sequence
$$\Sigma^{32}\TMF \ra (\Sigma^8\bo_1\cup \Sigma^{16}\bo_2)\sm \TMF\ra
\TMF(\Gamma_0(3)).$$ 
\end{prop}
The proof consists of computing the homotopy of the middle spectrum
and then applying the ``recognition principle''.  

\section{Some calculations in the building complex}

We want to understand how the map $\delta=q^*-f^*\colon \MF_*\ra
\MF_*(\Gamma_0(3))$ works, which in turn gives us a good understanding
of the corresponding map $\TMF\ra \TMF(\Gamma_0(3))$ in homotopy.

\begin{prop}
We have
$$\delta(\Delta^{2^r(2k+1)})=a_1^{3\cdot 2^{r+1}}a_3^{2^{r+1}(4k+1)} +
\text{higher terms in $a_1$} \pmod{2}.$$
\end{prop}
\begin{proof}
To get the formula, recall that 
$$f^*(\Delta)=-27\,a_3^4+a_1^3a_3^3,\qquad 
q^*(\Delta)=-19683\,a_3^4+2187\,a_1^3a_3^3-81\,a_1^6a_3^2+a_1^9a_3.$$
This gives 
$$f^*(\Delta)\equiv (a_3^4+a_1^3a_3^3),\qquad q^*(\Delta)\equiv
a_3^4+a_1^3a_3^3 +a_1^6a_3^2+a_1^9a_3 \pmod{2}.$$
Formally setting $a_1=t$ and $a_3=1$, the answer is obtained by
determining the first non-zero term in 
$$(1+t^3+t^6+t^9)^{2^r(2k+1)} - (1+t^3)^{2^r(2k+1)} \pmod {2}.$$
Since this expression is equal to
$(1+t^3)^{2^r(2k+1)}((1+t^6)^{2^r(2k+1)}-1)$, it is not hard to see
that the leading term is $t^{3\cdot 2^{r+1}}$.
\end{proof}

\begin{prop}
We have
$$\delta(c_4^k) =
2^{4+\nu_2(k)}\left[(\text{odd})\,a_1^{4(k-1)}\,a_1a_3+\text{higher
    terms in $a_1$} \right]$$
and
$$\delta(c_4^kc_6) =
2^3\left[(\text{odd})\,a_1^{4k+2}\,a_1a_3+\text{higher terms in $a_1$}
\right].$$
\end{prop}
\begin{proof}
The first equality is a straightforward application of the following
lemma to the identity $q^*(c_4)=f^*(c_4)+2^4(15\,a_1a_3)$.  The second
equality can be derived from the first together with the identity
$q^*(c_6)=f^*(c_6)+2^3(63\,a_1^2\,a_1a_3+756\,a_3^3)$. 
\end{proof}

\begin{lemma}
If $d>1$, then
$$(u+2^d\,v)^k = u^k+2^{d+\nu_2(k)}\,g(u,v),$$
where $g(u,v)\in \Z[u,v]$ has the form $g(u,v)=(\text{odd})\,
u^{k-1}v+\text{higher terms in $v$}$.
\end{lemma}


\begin{thebibliography}{GJM86}

\bibitem[AHR]{ando-hopkins-rezk-string-orientation}
Matthew Ando, Michael~J. Hopkins, and Charles Rezk, \emph{The string
  orientation of {$\mathrm{tmf}$}}, in preparation.

\bibitem[AHS01]{ando-hopkins-strickland-thm-of-cube}
M.~Ando, M.~J. Hopkins, and N.~P. Strickland, \emph{Elliptic spectra, the
  {W}itten genus and the theorem of the cube}, Invent. Math. \textbf{146}
  (2001), no.~3, 595--687. \MR{1 869 850}

\bibitem[Beh]{behrens-buildings-elliptic-curves}
Mark Behrens, \emph{Buildings, elliptic curves, and the {$K(2)$}-local sphere},
  arxiv:math.AT/0510026.

\bibitem[Beh06]{behrens-modular-description}
\bysame, \emph{A modular description of the {$K(2)$}-local sphere at the prime
  3}, Topology \textbf{45} (2006), no.~2, 343--402. \MR{MR2193339
  (2006i:55016)}

\bibitem[Del75]{deligne-formulaire}
P.~Deligne, \emph{Courbes elliptiques: formulaire d'apr\`es {J}. {T}ate},
  Modular functions of one variable, IV (Proc. Internat. Summer School, Univ.
  Antwerp, Antwerp, 1972), Springer, Berlin, 1975, pp.~53--73. Lecture Notes in
  Math., Vol. 476. \MR{MR0387292 (52 \#8135)}

\bibitem[DM82]{davis-mahowald-ext-a2-stunted-projective}
Donald~M. Davis and Mark Mahowald, \emph{Ext over the subalgebra {$A\sb{2}$} of
  the {S}teenrod algebra for stunted projective spaces}, Current trends in
  algebraic topology, Part 1 (London, Ont., 1981), CMS Conf. Proc., vol.~2,
  Amer. Math. Soc., Providence, RI, 1982, pp.~297--342. \MR{MR686123
  (85a:55018)}

\bibitem[GJM86]{goerss-jones-mahowald-generalized-bg-spectra}
Paul~G. Goerss, John D.~S. Jones, and Mark~E. Mahowald, \emph{Some generalized
  {B}rown-{G}itler spectra}, Trans. Amer. Math. Soc. \textbf{294} (1986),
  no.~1, 113--132. \MR{MR819938 (87d:55006)}

\bibitem[GM95]{gorbounov-mahowald-mo8}
Vassily Gorbounov and Mark Mahwolald, \emph{Some homotopy of the cobordism
  spectrum {$M{\rm O}\langle 8\rangle$}}, Homotopy theory and its applications
  (Cocoyoc, 1993), Contemp. Math., vol. 188, Amer. Math. Soc., Providence, RI,
  1995, pp.~105--119. \MR{MR1349133 (96i:55010)}

\bibitem[HBJ94]{hirzebruch-manifolds-mod-forms}
Friedrich Hirzebruch, Thomas Berger, and Rainer Jung, \emph{Manifolds and
  modular forms}, Vieweg, 1994, Translated by Peter S. Landweber.

\bibitem[HK01]{hu-kriz-real-oriented-homology}
Po~Hu and Igor Kriz, \emph{Real-oriented homotopy theory and an analogue of the
  {A}dams-{N}ovikov spectral sequence}, Topology \textbf{40} (2001), no.~2,
  317--399. \MR{MR1808224 (2002b:55032)}

\bibitem[HM]{hopkins-mahowald-from-elliptic-curves}
M.~J. Hopkins and M.~Mahowald, \emph{From elliptic curves to homotopy theory},
  1998 preprint, at http:{//}hopf.math.purdue.edu.

\bibitem[KM85]{katz-mazur-arithmetic-moduli}
Nicholas~M. Katz and Barry Mazur, \emph{Arithmetic moduli of elliptic curves},
  Annals of Mathematics Studies, vol. 108, Princeton University Press,
  Princeton, NJ, 1985. \MR{MR772569 (86i:11024)}

\bibitem[KW]{kitchloo-wilson-fibrations-real-spectra}
Nitu Kitchloo and W.~Stephen Wilson, \emph{On fibration related to real
  spectra}, to appear.

\bibitem[Sil86]{silverman-arith-ell-curves}
J.~H. Silverman, \emph{The arithmetic of elliptic curves}, Graduate Texts in
  Mathematics, Springer-Verlag, 1986.

\end{thebibliography}
\newcommand{\noopsort}[1]{} \newcommand{\printfirst}[2]{#1}
  \newcommand{\singleletter}[1]{#1} \newcommand{\switchargs}[2]{#2#1}
\providecommand{\bysame}{\leavevmode\hbox to3em{\hrulefill}\thinspace}
\providecommand{\MR}{\relax\ifhmode\unskip\space\fi MR }
\providecommand{\MRhref}[2]{%
  \href{http://www.ams.org/mathscinet-getitem?mr=#1}{#2}
}
\providecommand{\href}[2]{#2}

\end{document}